\documentclass[a4paper, 11pt]{article}
\RequirePackage[T1]{fontenc}
\RequirePackage{amsmath, amssymb, amsfonts, amsthm, amscd, mathtools, leftidx}
\RequirePackage{algorithm, algpseudocode}
\RequirePackage{enumerate}
\RequirePackage{multirow, booktabs}
\RequirePackage{pgfplots, pgfplotstable}
\RequirePackage{tikz}
\pgfplotsset{compat=1.18}
\usetikzlibrary{mindmap, shadows, calc}
\RequirePackage{caption, subfig}
\RequirePackage[hidelinks,pdfencoding=auto]{hyperref}
\RequirePackage{siunitx}
\RequirePackage{geometry}
\RequirePackage{setspace}
\theoremstyle{definition}
\newtheorem{definition}{Definition}

\newtheorem{remark}{Remark}
\newtheorem{example}{Example}
\theoremstyle{plain}
\newtheorem{theorem}{Theorem}

\newtheorem{proposition}{Proposition}
\newtheorem{corollary}{Corollary}

\DeclareMathOperator{\diag}{diag}
\DeclareMathOperator{\co}{co}
\DeclareMathOperator{\distr}{distr}

\newcommand{\RR}{\mathbb{R}}

\newcommand{\ZZ}{\mathbb{Z}}
\newcommand{\NN}{\mathbb{N}}

\newcommand{\low}[1]{\underline{#1}}
\newcommand{\up}[1]{\overline{#1}}

\def\tmatrices{\mathcal{P}}

\def\states{\mathcal{X}}

\def\cset{\mathcal{M}}
\def\csete{\mathcal{C}}

\title{Reversible Imprecise Markov Chains}
\author{Damjan \v{S}kulj\\
	University of Ljubljana, Faculty of Social Sciences\\
	Kardeljeva pl. 5, SI-1000 Ljubljana, Slovenia\\
	\href{mailto:damjan.skulj@fdv.uni-lj.si}{\tt damjan.skulj@fdv.uni-lj.si}}
\date{}

\begin{document}
	\maketitle
	\begin{abstract}
		Reversible Markov chains play a central role in stochastic modelling and in algorithms such as Markov chain Monte Carlo (MCMC). Motivated by the fundamental importance of reversibility in classical settings, this paper develops a theoretical framework for reversible imprecise Markov chains. We focus on their structural properties and their representation through joint distribution matrices. Adopting the strong independence interpretation, we reverse every precise chain compatible with a given imprecise Markov chain specification. Since the reversed ensemble generally cannot be encoded by the usual forward model defined by an imprecise initial distribution and a set of transition matrices, we introduce a symmetric representation based on credal sets of two-step joint distribution (or edge measure) matrices. This strictly more expressive framework naturally admits the reversal operation and reduces reversibility to simple matrix symmetry. Moreover, forward and reverse dynamics can be described simultaneously within a single closed convex set, providing a unified structural basis for the analysis of expectations of path-dependent functionals. We illustrate the theory with random walks on graphs and outline methods for computing lower and upper expectations of such functionals.

		\smallskip
		
		
	\end{abstract}
	
	\section{Introduction}
	%========================================================
	Classical Markov chain analysis assumes that the initial distribution and each transition matrix are known exactly. In practise, these assumptions are rarely fulfilled due to vague expert judgements, scarce data, and unknown dependencies. Replacing the precise initial and transition probabilities with \emph{credal sets} leads to more robust models and, consequently, to more reliable inferences. This idea was first introduced as \emph{Markov set chains}~\cite{hart:98}. Subsequent work in the general framework of \emph{imprecise probabilities} \cite{augustin2014introduction, walley:91} developed analogous models for discrete time~\cite{krein:03,decooman-2008-a, utkin:02, krak2019hitting,  lopatatzidis2017computing,lopatatzidis2016robust, skulj:09IJAR, t2021average} and later extended them to continuous time~\cite{erreygers2022markovian, krak2017imprecise, nendel2021markov, skulj:15AMC}. The resulting theory, known as \emph{imprecise Markov chains}, enables exact worst-case inference without assuming precise parameters.
	
	In this paper, we extend this theory to the \emph{time reversal} of imprecise Markov chains, that is, to studying the behaviour of a process when the direction of time is reversed. This work builds upon two earlier conference	publications: \cite{skulj2025}, which introduced initial results on time reversal in the imprecise setting, and \cite{vskulj2025random}, which analysed	weighted random walks with interval weights as a motivating example for	reversible imprecise Markov chains and developed some preliminary general	results. In a reversed Markov chain, the dependency structure—such as the
	Markov property—acquires a new interpretation: instead of predicting the	future from the past, one infers the past from the future. A chain is therefore	\emph{reversible} if its stochastic law remains unchanged under time reversal.
	
	Reversible chains are fundamental to Monte Carlo methods, especially MCMC, because imposing a detailed balance with a chosen target distribution provides a direct, easily verifiable guarantee that the target is stationary~\cite{green1995reversible,hastings1970monte,jerrum1996markov}. The same property underlies the key results of queueing theory~\cite{Asmussen2003}, network theory~\cite{kelly2011reversibility}, and the applications in statistical physics and information geometry~\cite{WolferWatanabe2021}.

	In the present work we do not directly address algorithmic applications but rather develop a general theoretical framework that characterises when an imprecise process can be considered reversible and how such processes can be represented through joint distribution matrices.
	We construct the time reversal of imprecise Markov chains under \emph{strong independence}, and leave the more general case of \emph{epistemic irrelevance} for future work. Strong independence allows to describe an imprecise chain as a \emph{compatible set} of \emph{inhomogeneous precise chains}. We reverse each precise chain individually and obtain the reversed imprecise process by collecting the reversals of all members of the compatible set.
	
	Although this is possible in principle, several challenges arise when describing the reversed processes. Using an example, we show that the usual \emph{forward representation}, which propagates a set of initial distributions through a set of transition matrices, is unsuitable for inversion. Instead, we introduce a symmetric view based on sets of two-step \emph{joint distribution} (or \emph{edge-measure}) matrices. This representation naturally takes inversion into account and is strictly more expressive than the forward view. Another advantage is that transition probabilities (obtained by conditioning on marginals) are never explicitly required and thus no assumption of strictly positive marginals is needed. Although in practise transition probabilities are the most natural modelling primitives, the joint distribution formulation provides a mathematically symmetric representation that avoids assuming positive marginals or invertible transitions. This allows time reversal and reversibility to be expressed directly in terms of observable joint behaviour rather than conditional probabilities.
	
	While the main goal of this paper is to establish the theoretical foundations of time reversal and reversibility for imprecise Markov chains, we also discuss their numerical aspects. Each collection of joint matrices forms a credal set, and this structure is preserved under reversal. Consequently, evaluating lower and upper expectations of multi-step path functionals can be formulated as a \emph{multilinear programming} problem. We outline how this general formulation includes the two-step case as a special instance and, in principle, extends to longer paths, although with rapidly increasing computational complexity.
	
	The rest of the paper is organised as follows: Section~\ref{s-pn} recalls the notation and tools needed to analyse reversed imprecise Markov chains, including the reversal of inhomogeneous precise chains. Section~\ref{s-jdmmc} introduces the symmetric modelling approach based on joint distribution matrices for both homogeneous and inhomogeneous chains. Section~\ref{s-imc} gives an overview of imprecise Markov chains, outlines the three most common interpretations, and then focuses on strong independence, which forms the basis for the rest of the paper. Section~\ref{s-rimc} examines reversed imprecise Markov chains, represented by sets of joint distribution matrices, and derives matrix-based conditions for reversibility. Section~\ref{s-rwwg} applies the general theory to an important special case of random walks on weighted graphs, whose weights are now allowed to be given by intervals. Section~\ref{s-nc} discusses how expectations with respect to joint distributions can be estimated in practise as optimisation problems. Finally, Section~\ref{s-co} summarises our contributions and points out directions for future work.

	\section{Preliminaries and notation}\label{s-pn}
	
	\subsection{Notation}
	Let \(\states\) be a non-empty finite set, the \emph{state space} of 
	Markov chains considered, and let \(|\states|\) denote its cardinality.
	By \(\NN\) we denote the set of natural and by 
	\(\RR\) the set of real numbers.
	
	All vectors are assumed to be \emph{row} vectors, typically denoted by
	\(q,r,\dots\) and interpreted as probability mass functions. Their
	components are thus non-negative: we write \(q\ge 0\) for component-wise
	non-negativity and \(q>0\) when every component is strictly positive.
	Further, \(q^{\top}\) denotes the transposed vector.  Unless stated otherwise, all vectors have dimension \(|\states|\)
	and can be regarded as real-valued functions
	\(\states\to\RR\).  The constant-one vector of length \(|\states|\) is
	denoted \(\mathbf 1\).
	
	Capital letters \(P,Q,\dots\) will denote
	\(|\states|\!\times\!|\states|\) matrices, with \(P(x,y)\) the element in
	row \(x\) and column \(y\).  The diagonal matrix with diagonal elements \(q\)  is denoted by \(\diag(q)\).  
	A matrix \(P\) is \emph{stochastic} when its
	rows are probability vectors, i.e.\ \(P\mathbf 1^{\top}=\mathbf
	1^{\top}\).
	
	For \(m,n\in\NN\) we write \(X_{m:n}:=(X_m,\dots,X_n)\); if \(n<m\) this
	means the reversed sequence \((X_m,X_{m-1},\dots,X_n)\).  Throughout we
	consider only finite segments of a chain.  Unless a different start time
	is important, we simply write \(X_{1:N}\) for some \(N\in\NN\).

	\subsection{Markov chains}
	Now we review basic properties of Markov chains. For more detailed introduction, reader is referred to classical textbooks, e.g.,  \cite{KemenySnell1976, norris:97}.  
	
	\subsubsection{Basic model}
	A sequence of random variables \((X_1,X_2,\dots)\) taking values in the
	finite state space \(\states\) is a \emph{discrete-time Markov chain} if,
	for every \(n\in\NN\) and \(x_1,\dots,x_{n+1}\in\states\),
	\begin{align}
		\label{eq:MC}
		\Pr\bigl(X_{n+1}=x_{n+1}\mid X_n=x_n,\dots,X_1=x_1\bigr)
		\;=\;
		\Pr\bigl(X_{n+1}=x_{n+1}\mid X_n=x_n\bigr)
		\;=:\;
		P_n(x_n,x_{n+1}),
	\end{align}
	where \(P_n(x,y)\) is the \emph{one-step transition probability}
	from state \(x\) to state \(y\) at time step \(n\).
	Collecting these probabilities yields the row-stochastic
	\(|\states|\!\times\!|\states|\) \emph{transition matrix} \(P_n\).
	
	The chain is \emph{time-homogeneous} when \(P_n = P\) for all
	\(n\in\NN\), where $P$ is fixed transition matrix.  
	Let \(q^{(1)}\) denote the \emph{initial distribution},
	\(q^{(1)}(x)=\Pr(X_1=x)\).
	For a homogeneous chain the distribution at time \(n\ge 1\) is
	\[
	q^{(n)} \;=\; q^{(1)} P^{\,n-1},
	\]
	while for an inhomogeneous chain it is
	\(q^{(n)} = q^{(1)} P_1P_2\cdots P_{n-1}\).
	
	Although Markov chains are often analysed on an infinite time horizon,
	our focus will be on finite segments \((X_1,\dots,X_N)\), for which the
	transition matrices \(P_1,\dots,P_{N-1}\) may vary with time.

	%--------------------------------------------------------
	\subsubsection{Stationary time-homogeneous chains}
	A time-homogeneous chain is \emph{stationary} if the marginal
	distribution of \(X_n\) is the same for every \(n\).
	Equivalently, at every time $n$ the distribution \(\pi\) of $X_n$ is an \emph{invariant} distribution for the
	transition matrix \(P\): 
	\[
	\pi P \;=\; \pi .
	\]
	On a finite state space at least one invariant distribution exists.  If
	\(P\) is also \emph{ergodic}—that is, irreducible and aperiodic—the
	invariant distribution is unique and strictly positive, and the law of
	\(X_n\) converges to \(\pi\) regardless of the initial state.
	With fixed marginal distributions, attention naturally shifts from the analysis of marginal distributions
	to the joint probability law.
	
	The notion of stationarity extends to \emph{imprecise} Markov chains, where the credal sets corresponding to marginal distributions 
	are invariant under	the imprecise transition operator~\cite{decooman-2008-a,skulj:09IJAR}.
	
	%--------------------------------------------------------
	\subsubsection{Reversible Markov chains}
	Let \((X_{1:N})\) be a time-homogeneous Markov chain with transition
	matrix \(P\) and stationary distribution \(\pi\).
	The reversed sequence \((X_{N:1})\) is itself again a Markov chain, with
	transition probabilities
	\begin{equation}
		\label{eq:reverse-transition}
		P^{*}(x,y)\;=\;\frac{\pi(y)\,P(y,x)}{\pi(x)},
	\end{equation}
	or, in matrix form,
	\(P^{*}=\operatorname{diag}(\pi)^{-1} P^{\top}\operatorname{diag}(\pi)\).
	
	The reversed chain shares the stationary distribution \(\pi\).  This
	construction assumes \(\pi(x)>0\) for all \(x\).  If the original chain is irreducible and aperiodic, these properties are preserved under reversal.
	If \(P^{*}(x,y)=P(x,y)\) for all \(x,y\in\states\), the chain is
	said to be \emph{reversible}.  From~\eqref{eq:reverse-transition}, we simply deduce that reversibility is
	equivalent to the \emph{detailed-balance} condition
	\begin{equation}
		\label{eq:detailed-balance}
		\pi(x)\,P(x,y)=\pi(y)\,P(y,x)
		\qquad(x,y\in\states).
	\end{equation}
	
	Reversibility can also be characterised by symmetry of path
	probabilities (cf.\ \cite[Chapter 1.2]{kelly2011reversibility}).
	
	\begin{theorem}\label{thm:rev-equiv}
		Let \((X_n)_{n\in\NN}\) be a \emph{stationary} Markov chain on the finite
		state space \(\states\) with transition matrix \(P\) and stationary
		distribution \(\pi\).
		The following statements are equivalent:
		\begin{enumerate}[(i)]
			\item \textbf{Detailed balance:}
			\(\pi(x)\,P(x,y)=\pi(y)\,P(y,x)\) for all \(x,y\in\states\).
			\item \textbf{Two-step symmetry:}
			\(\Pr(X_1=x_1,X_2=x_2)=\Pr(X_1=x_2,X_2=x_1)\)
			for all \(x_1,x_2\in\states\).
			\item \textbf{Finite-path symmetry:}
			\(\Pr(X_{1:N}=x_{1:N})=\Pr(X_{N:1}=x_{1:N})\)
			for every \(N\in\NN\) and every \(x_{1:N}\in\states^{N}\).
		\end{enumerate}
		Hence the chain is reversible if and only if any (and therefore all) of
		these conditions hold.
	\end{theorem}

	\subsection{Time reversal for time-inhomogeneous chains}
	
	The reverse of a stationary Markov chain does not depend on a specific starting point---therefore we can typically assume unlimited sequences. 
	However, for non-stationary or inhomogeneous processes, reversal is only meaningful within finite time horizon, as the terminal distribution is needed to calculate the reversed transitions. Given a finite-horizon chain \((X_{1:N})\), its reversed chain is
	\((X_{N:1})\), and although technically this is merely a re-indexing, it is important to analyse which structural properties, such as the Markov property, carry over.
	
	To extend reversal to the imprecise setting we work with
	\emph{time-inhomogeneous} Markov chains and therefore restrict attention to finite
	horizons \(N\).  For \(k=1,\dots,N-1\) let \(P_{k}\) denote the transition matrix at time
	\(k\).  With initial distribution \(q^{(1)}\) the marginal at time \(k\)
	is
	\begin{equation}\label{eq:dist-at-k}
		q^{(k)} \;=\; q^{(1)} P_{1}\cdots P_{k-1},
		\qquad k=1,\dots,N.
	\end{equation}
	We call \(\Gamma=(q^{(1)},P_{1:N-1})\) the \emph{transition law} of
	\((X_{1:N})\).
	
	The detailed balance condition does not extend directly to the inhomogeneous case, but
	we can still define a \emph{time-reversed transition law}.  Throughout
	this subsection all probability mass functions are assumed to have full
	support, i.e.\ \(q(x)>0\) for every \(x\in\states\).
	
	\begin{definition}[$q$-reverse]\label{def:q-reverse}
		For a stochastic matrix \(P\) and a strictly positive mass function
		\(q\), the \emph{$q$-reverse} of \(P\) is
		\[
		P^{*}_{q}\;:=\;\operatorname{diag}(qP)^{-1}\,P^{\top}\,
		\operatorname{diag}(q).
		\]
	\end{definition}
	
	If \(\pi\) is a stationary distribution of \(P\), then
	\(P^{*}_{\pi}=P^{*}\), the usual homogeneous time reverse.
	
	\begin{proposition}\label{prop:q-reverse-stochastic}
		Let \(P\) be a stochastic matrix and \(q>0\) with \(qP>0\).
		Then \(P^{*}_{q}\) is stochastic.
	\end{proposition}
	
	\begin{proof}
		All matrices are well defined and non-negative.  Moreover,
		\[
		P^{*}_{q}\mathbf 1^{\top}
		\;=\;
		\operatorname{diag}(qP)^{-1}P^{\top}\operatorname{diag}(q)\mathbf 1^{\top}
		\;=\;
		\operatorname{diag}(qP)^{-1}(qP)^{\top}
		\;=\;
		\mathbf 1^{\top},
		\]
		so each row of \(P^{*}_{q}\) sums to one.
	\end{proof}
	
	The next result shows that \(P^{*}_{q}\) maps the image of \(q\) back to
	\(q\) and that the reversal operation is an involution.
	
	\begin{proposition}\label{prop:q-reverse-involution}
		For a stochastic matrix \(P\) and \(q>0\) with \(qP>0\):
		\begin{enumerate}[(i)]
			\item \((qP)\,P^{*}_{q}=q\);
			\item \(\bigl(P^{*}_{q}\bigr)^{*}_{\,qP}=P\).
		\end{enumerate}
	\end{proposition}
	
	\begin{proof}
		(i) Using Definition~\ref{def:q-reverse},
		\[
		qPP^{*}_{q}
		\;=\;
		(qP)\operatorname{diag}(qP)^{-1}P^{\top}\operatorname{diag}(q)
		\;=\;
		\mathbf 1\,P^{\top}\operatorname{diag}(q)
		\;=\;
		(P\mathbf 1^{\top})^{\!\top}\operatorname{diag}(q)
		\;=\;
		q.
		\]
		
		(ii) By (i) we have \(qPP^{*}_{q}=q\).  Hence
		\[
		(P^{*}_{q})^{*}_{\,qP}
		\;=\;
		\operatorname{diag}(qPP^{*}_{q})^{-1}
		(P^{*}_{q})^{\top}\operatorname{diag}(qP)
		\;=\;
		\operatorname{diag}(q)^{-1}
		\bigl(\operatorname{diag}(qP)^{-1}P^{\top}\operatorname{diag}(q)\bigr)^{\top}
		\operatorname{diag}(qP)
		\;=\;
		P.
		\]
	\end{proof}
	
	\begin{definition}[Reverse transition law]\label{def:rev-law}
		Let \(\Gamma=(q^{(1)},P_{1:N-1})\) be a transition law with \(q^{(1)}>0\)
		and \(q^{(k)}>0\) for all \(k\) as in~\eqref{eq:dist-at-k}.  Define
		\[
		\Gamma^{*}:=\bigl(q^{(N)},\,P^{*}_{N-1:1}\bigr),
		\qquad
		P^{*}_{k}:=(P_{k})^{*}_{\,q^{(k)}}
		=\operatorname{diag}\!\bigl(q^{(k+1)}\bigr)^{-1}
		P_{k}^{\top}\operatorname{diag}\!\bigl(q^{(k)}\bigr).
		\]
		We call \(\Gamma^{*}\) the \emph{reverse transition law} of \(\Gamma\).
	\end{definition}
	
	\begin{remark}
		For \(N=2\) we have \(\Gamma=(q,P)\) and
		\(\Gamma^{*}=(qP,P^{*}_{q})\).
	\end{remark}
	
	\begin{theorem}\label{thm:inhomog-reversal}
		Let \((X_{1:N})\) be an inhomogeneous Markov chain with transition law
		\(\Gamma=(q^{(1)},P_{1:N-1})\) satisfying \(q^{(k)}>0\) for every
		\(k\).  Then the reversed chain \((X_{N:1})\) is Markov with transition
		law \(\Gamma^{*}\) of Definition~\ref{def:rev-law}.
	\end{theorem}
	
	\begin{proof}
		For \(k=2,\dots,N\) and states \(x_{k-1},x_{k}\), we have
		\begin{align*}
		\Pr(X_{k-1}=x_{k-1}\mid X_{k:N}=x_{k:N}) 
		& \;=\;
		\frac{\Pr(X_{k-1:N}=x_{k-1:N})}{\Pr(X_{k:N}=x_{k:N})} \\
		& \;=\;
		\frac{q^{(k-1)}(x_{k-1}) P_{k-1}(x_{k-1}, x_k) \cdots P_{N-1}(x_{N-1}, x_N)}{q^{(k)}(x_{k}) P_{k}(x_{k}, x_{k+1}) \cdots P_{N-1}(x_{N-1}, x_N)}\\
		& \;=\;		
		\frac{P_{k-1}(x_{k-1},x_{k})\,q^{(k-1)}(x_{k-1})}{q^{(k)}(x_{k})}
		\;=\;
		P^{*}_{k-1}(x_{k},x_{k-1}),
		\end{align*}
		which is precisely the \((k-1)\)-st reversed matrix and shows that the reversed process satisfies the Markov property. Multiplying these
		matrices yields the joint distribution
		\[
		\Pr(X_{N:1}=x_{N:1})
		\;=\;
		q^{(N)}(x_{N})\,
		P^{*}_{N-1}(x_{N},x_{N-1})\cdots P^{*}_{1}(x_{2},x_{1}),
		\]
		so \((X_{N:1})\) obeys the transition law \(\Gamma^{*}\).
	\end{proof}
	
	\begin{remark}\label{rem:reversal-notes}
		\mbox{} 
		\begin{enumerate}
			\item Each reversed matrix \(P^{*}_{k}\) depends on the initial
			distribution \(q^{(1)}\) (via \(q^{(k)}\)) as well as on the
			forward matrices \(P_{1},\dots,P_{k}\).
			\item Even if the forward chain is homogeneous (\(P_{k}\equiv P\)), the
			reversed chain is generally \emph{inhomogeneous} unless
			\(q^{(1)}\) equals the stationary distribution \(\pi\) of \(P\);
			in that case \(P^{*}_{k}\equiv P^{*}_{\pi}=P^{*}\) and the
			reversed chain is homogeneous.
		\end{enumerate}
	\end{remark}

\section{Joint-Distribution Matrices for Markov Processes}\label{s-jdmmc}

For a \emph{precise} stationary Markov chain, reversibility is
characterised by two equivalent conditions: the
\emph{detailed-balance} equalities~\eqref{eq:detailed-balance} and the
\emph{two-step path symmetry} of the chain
(Theorem~\ref{thm:rev-equiv}).  
Our aim is to extend this notion to \emph{imprecise} Markov chains.
Detailed balance ties the stationary distribution at each ordered pair
of states to the corresponding transition probabilities and therefore
does not generalise naturally to credal sets, where such pointwise
constraints are too restrictive.  
By contrast, path symmetry translates to two-step credal sets directly by requiring that they remain invariant under swapping its
coordinates.

\subsection{Definition and basic identities}

\begin{definition}[Joint distribution (edge-measure) matrix]
	Let \(q\) be a probability mass function on the finite state space
	\(\states\) and \(P\) a stochastic transition matrix.
	The \emph{joint distribution matrix} associated with the two-step chain
	\((X_1,X_2)\) is
	\[
	Q_{q,P}(x,y):=q(x)\,P(x,y), \qquad x,y\in\states .
	\]
\end{definition}

When the context is clear we simply write \(Q\).
If \(q>0\) then \(P\) can be recovered from \(Q\) via
\(P(x,y)=Q(x,y)/q(x)\).

In the theory of reversible chains the matrix
\(Q(x,y)=\pi(x)P(x,y)\) is standard; it is often called the
\emph{edge measure}~\cite{LevinPeresWilmer2017,WolferWatanabe2021} or, in
the electrical-network analogy, the \emph{conductance} (or
\emph{flow}) matrix~\cite{DoyleSnell1984}.
We adopt this joint matrix representation in place of the classical
forward pair \((q,P)\); it proves more convenient for studying
reversibility and extends naturally to the imprecise setting.

\begin{proposition}\label{prop:q-matrix}
	Let \(Q=Q_{q,P}\) with \(q>0\).  Then
	\begin{enumerate}[(i)]
		\item \(Q=\operatorname{diag}(q)\,P\);
		\item \(P=\operatorname{diag}(q)^{-1}\,Q\);
		\item \(q=\mathbf 1\, Q^\top\) \quad (row sums);
		\item \(qP=\mathbf 1\,Q\) \quad (column sums);
		\item the joint matrix for the reversed pair \((X_2,X_1)\) is \(Q^{\top}\).
	\end{enumerate}
\end{proposition}

Whenever \(q>0\) the mapping \((q,P)\mapsto Q\) in
Proposition~\ref{prop:q-matrix}\,(i) is bijective, because
\(\operatorname{diag}(q)^{-1}\) exists and retrieves \(P\) from \(Q\).
If some component \(q(x)\) equals zero, uniqueness fails: the \(x\)-row
of \(Q\) is identically zero, so any stochastic vector may be chosen for
the \(x\)-row of \(P\) without affecting \(Q\).
Nevertheless, for \emph{every} joint matrix \(Q\) one can find at least
one pair \((q,P)\) with \(Q=\operatorname{diag}(q)P\); for instance take
\(q=Q\,\mathbf 1^{\top}\) and define
\[
P(x,\cdot)=
\begin{cases}
	Q(x,\cdot)/q(x), & q(x)>0,\\
	\text{any stochastic vector}, & q(x)=0.
\end{cases}
\]
Thus the forward and joint representations coincide exactly when
\(q\) has full support; otherwise the forward pair is a
(non-unique) refinement of \(Q\).

\begin{corollary}
	For a stochastic matrix \(P\) with a unique stationary distribution
	\(\pi>0\), the following are equivalent:
	\begin{enumerate}[(i)]
		\item \(P\) is reversible with respect to \(\pi\);
		\item the joint matrix \(Q_{\pi,P}\) is symmetric
		\(\bigl(Q_{\pi,P}=Q_{\pi,P}^{\top}\bigr)\).
	\end{enumerate}
\end{corollary}

Hence a precise chain is reversible \emph{iff} its joint distribution
matrix is symmetric.

\subsection{Application to inhomogeneous chains}

We now extend the joint distribution representation to finite-horizon,
not necessarily time-homogeneous, Markov chains of arbitrary length.
Strict positivity of the marginal distributions is \emph{not} required.

\begin{proposition}\label{prop:p_Gamma}
	Let \(\Gamma=(q^{(1)},P_{1:N-1})\) be the transition law of a
	time-inhomogeneous Markov chain \((X_{1:N})\) on a finite state
	space~\(\states\).
	For \(k=1,\dots,N-1\) set
	\(q^{(k)}:=q^{(1)}P_{1}\cdots P_{k-1}\) and define
	\[
	Q_{k}:=\operatorname{diag}\!\bigl(q^{(k)}\bigr)\,P_{k}.
	\]
	Then, for every path \(x_{1:N}\in\states^{N}\),
	\begin{equation}\label{eq:jpf-factor}
		p_{\Gamma}(x_{1:N})
		\;=\;
		\frac{\displaystyle\prod_{k=1}^{N-1}Q_{k}(x_{k},x_{k+1})}
		{\displaystyle\prod_{k=2}^{N-1}q^{(k)}(x_{k})},
	\end{equation}
	with the convention \(0/0:=0\).  Moreover,
	\begin{equation}\label{eq:qk-marginals}
		\mathbf 1\, Q_{k}^\top=q^{(k)},
		\qquad
		\mathbf 1\,Q_{k}=q^{(k+1)},
		\qquad k=1,\dots,N-1 .
	\end{equation}
\end{proposition}

\begin{proof}
	\emph{Marginal identities.}
	Because each \(P_{k}\) is stochastic
	(\(\mathbf 1 P_{k}^{\top}=\mathbf 1\)),
	\[
	\mathbf 1 \, Q_{k}^{\top}
	=\operatorname{diag}\!\bigl(q^{(k)}\bigr)\mathbf 1\,P_{k}^\top
	=q^{(k)},
	\quad
	\mathbf 1\,Q_{k}
	=\mathbf 1\operatorname{diag}\!\bigl(q^{(k)}\bigr)P_{k}
	=q^{(k)}P_{k}=q^{(k+1)},
	\]
	which proves~\eqref{eq:qk-marginals}.
	
	\smallskip
	\noindent\emph{Factorisation.}
	The standard product rule gives
	\begin{equation}\label{eq:p_Gamma}
		p_{\Gamma}(x_{1:N})
		\;=\;
		q^{(1)}(x_{1})
		\prod_{k=1}^{N-1}P_{k}(x_{k},x_{k+1}).
	\end{equation}
	For each \(k\) with \(q^{(k)}(x_{k})>0\),
	\(P_{k}(x_{k},x_{k+1})=
	Q_{k}(x_{k},x_{k+1})/q^{(k)}(x_{k})\).
	If \(q^{(k)}(x_{k})=0\) the entire path probability is zero, which we
	interpret as \(Q_{k}(x_{k},x_{k+1})/q^{(k)}(x_{k})=0\).
	Substituting these fractions into~\eqref{eq:p_Gamma} and cancelling
	\(q^{(1)}(x_{1})\) yields~\eqref{eq:jpf-factor}.
\end{proof}

\begin{corollary}\label{cor:p_Gamma-reverse-pos}
	Assume the forward law
	\(\Gamma=(q^{(1)},P_{1:N-1})\) satisfies \(q^{(k)}>0\) for
	\(k=1,\dots,N\).
	Let \(\Gamma^{*}=(q^{(N)},P^{*}_{N-1:1})\) be the reverse law and put
	\(Q_{k}=\operatorname{diag}\!\bigl(q^{(k)}\bigr)P_{k}\).
	Then, for every path \(x_{1:N}\in\states^{N}\),
	\[
	p_{\Gamma^{*}}(x_{1:N})
	=\frac{\displaystyle\prod_{k=1}^{N-1}Q_{k}^{\top}(x_{k},x_{k+1})}
	{\displaystyle\prod_{k=2}^{N-1}q^{(k)}(x_{k})}.
	\]
\end{corollary}

\begin{proof}
	Because each \(q^{(k)}\) is positive, every
	\(\operatorname{diag}\!\bigl(q^{(k)}\bigr)\) is invertible, and
	\(P^{*}_{k}=(P_{k})^{*}_{q^{(k)}}\) is well defined.  Applying
	Proposition~\ref{prop:p_Gamma} to \(\Gamma^{*}\) and Proposition~\ref{prop:q-matrix}(v), which gives
	\(Q^{*}_{k}=Q_{k}^{\top}\), completes the proof.
\end{proof}

\smallskip\noindent
A single time step of a (possibly inhomogeneous) Markov chain can thus
be described either by the forward pair \((q,P)\) or by its
joint distribution matrix \(Q=\operatorname{diag}(q)\,P\).
Consequently, an \(N\)-step transition law can be represented by a
\emph{sequence} of joint matrices
\(Q_{1},Q_{2},\dots,Q_{N-1}\).
These matrices cannot be chosen arbitrarily: the marginal distributions
of adjacent matrices must coincide.

\begin{definition}[Marginal compatibility]
	Joint distribution matrices \(Q_{1}\) and \(Q_{2}\) on \(\states\) are
	\emph{marginally compatible} if
	\[
	\mathbf 1\,Q_{1} \;=\; \mathbf 1 Q_{2}^{\top},
	\]
	i.e.\ the marginal of the second coordinate under \(Q_{1}\) equals the
	marginal of the first coordinate under \(Q_{2}\).
\end{definition}

Working with joint matrices removes the need for inverting
\(\operatorname{diag}(q^{(k)})\); hence, strict positivity matters only when a
one-to-one correspondence with the forward description is desired.

\begin{corollary}\label{cor:gamma-q-equivalence}
	Assume all marginals of \(\Gamma\) in
	Proposition~\ref{prop:p_Gamma} are strictly positive.
	Then
	\[
	\Gamma \;\longmapsto\; Q_{1:N-1}
	\]
	is a bijection between transition laws with positive marginals and
	sequences of joint matrices \(Q_{1:N-1}\) that are pairwise marginally
	compatible and have positive row and column sums.
\end{corollary}

\begin{proof}
	Forward direction: uniqueness and marginal compatibility follow directly from
	Proposition~\ref{prop:p_Gamma}.
	For the converse, let \((Q_{1},\dots,Q_{N-1})\) satisfy the stated
	conditions.  Define \(q^{(k)}:=\mathbf 1\, Q_{k}^\top > 0\)  and
	\(P_{k}:=\operatorname{diag}\!\bigl(q^{(k)}\bigr)^{-1}Q_{k}\); this yields
	a unique transition law \(\Gamma=(q^{(1)},P_{1:N-1})\).
\end{proof}

\begin{corollary}\label{cor:reverse-Q}
	Let \(\Gamma=(q^{(1)},P_{1:N-1})\) correspond to
	\(Q_{1:N-1}\) as in Corollary~\ref{cor:gamma-q-equivalence}.
	The reversed law
	\(\Gamma^{*}=(q^{(N)},P^{*}_{N-1:1})\)
	corresponds to the \emph{transposed} sequence
	\(Q^{\top}_{N-1:1}\), which is again marginally compatible.
\end{corollary}

Thus every (possibly inhomogeneous) transition law \(\Gamma\) can be
encoded by a marginally compatible sequence of joint distribution
matrices \(Q_{1:N-1}\).  
Time reversal is achieved simply by reversing the sequence and
transposing each matrix:
\[
\Gamma_{Q_{1:N-1}}^{*} \;=\; \Gamma_{Q_{N-1:1}^{\top}}.
\]

\section{Imprecise Markov chains}\label{s-imc}

\subsection{Modelling imprecision in probability distributions}

\emph{Imprecise probabilities} \cite{augustin2014introduction, walley:91}
represent uncertainty by \emph{sets} of probability distributions rather
than a single “precise” mass function.  On a finite state space
\(\states\), an imprecise distribution is any non-empty set
\[
\csete \;\subseteq\;
\Bigl\{\,p\mid \states\!\to\![0,1]\;\bigm|\;
\textstyle\sum_{x\in\states}p(x)=1\Bigr\}.
\]
We reserve the symbol \(\cset\) for \emph{credal sets}, namely those
\(\csete\) that are both \emph{closed} and \emph{convex}.  Credal sets
allow convenient algebraic representations and dual characterisations.

Any (not necessarily convex) set \(\csete\) of probability mass
functions induces the \emph{lower} and \emph{upper expectations}
\[
\low E_{\csete}(f) \;=\; \inf_{p\in\csete}E_p[f],
\qquad
\up E_{\csete}(f) \;=\; \sup_{p\in\csete}E_p[f],
\]
where \(f:\states\!\to\!\RR\) is a \emph{gamble} (real-valued functional) and
\(E_p[f]=\sum_{x\in\states}p(x)f(x)\). The $\inf$ and $\sup$ clearly become $\min$ and $\max$ respectively if the set $\csete$ is closed, such as is the case with a credal set. In the literature these
functionals are also called \emph{lower} and \emph{upper previsions}
\cite{miranda2008}.

For a credal set \(\cset\) the lower expectation can be evaluated
efficiently by linear programming, and the mapping
\(f\mapsto\low E_{\cset}(f)\) \emph{uniquely characterises}
\(\cset\):
\[
\cset
\;=\;
\Bigl\{\,p \mid E_p[f]\ge \low E_{\cset}(f)\text{ for all }f\Bigr\}.
\]

\subsection{Modelling imprecision in Markov chains}

A classical (precise) Markov chain on a finite state space \(\states\) is
fully specified by an initial distribution \(q^{(1)}\) on \(\states\)
and a single transition matrix \(P\).
To account for parameter uncertainty, an \emph{imprecise Markov chain}
replaces the single matrix by a \emph{set} of row-stochastic matrices
\[
\tmatrices \;\subseteq\;
\{\,P:\states\times\states\!\to\![0,1] \mid
P\mathbf 1^{\top}=\mathbf 1^{\top}\,\},
\]
usually assumed to be closed and convex.
Likewise, the initial distribution is assumed to lie in a credal set
\(\cset\).
The true transition matrix is unknown but assumed to belong to
\(\tmatrices\).

Throughout this paper, an \emph{imprecise Markov chain} refers to a collection of stochastic processes whose one-step transitions (at every time $n$) are selected—possibly depending on time or history—from $\tmatrices$.
The exact nature of this collection depends on the adopted interpretation of imprecision each imposing different assumptions on how the matrices are selected over time.

Three principal interpretations appear in the literature:

\smallskip
\noindent\textbf{Sensitivity analysis}
\cite{krein:03,utkin:02}.  
The chain is time-homogeneous with some unknown
\(P\in\tmatrices\); the compatible processes are precisely these
homogeneous chains.

\smallskip
\noindent\textbf{Strong independence} \cite{skulj:09IJAR}.  
At each time \(n\) the transition matrix \(P_{n}\) may vary
arbitrarily within \(\tmatrices\) but is independent of the past
beyond \(X_{n}\).  
A law is specified by a sequence
\((P_{1:N-1})\in\tmatrices^{N-1}\), so the compatible processes are
time-inhomogeneous Markov chains.

\smallskip
\noindent\textbf{Epistemic irrelevance} \cite{decooman-2008-a}.  
At time \(n\) the choice of transition matrix may depend on the entire
history \(x_{1:n}\); for each history one selects a probability vector
from a set \(\tmatrices_{x_{n}}\) that governs \(X_{n+1}\).
The resulting family is even larger and is no longer Markovian
in general.

These models satisfy the strict inclusions
\[
\{\text{Sensitivity}\}\;\subset\;
\{\text{Strong independence}\}\;\subset\;
\{\text{Epistemic irrelevance}\},
\]
and each yields a set \(\distr(X_{1:N})\) of joint distributions
on finite paths, allowing inferences about marginals or other
functionals via lower and upper expectations.

A key computational requirement is that \(\tmatrices\) have
\emph{convex, separately specified rows}: each row
\(P(\cdot\mid x)\) can be selected independently from a convex set
\(\tmatrices_{x}\).
This property guarantees that the induced marginal credal sets are
convex and that lower and upper expectations remain tractable under
both strong independence and epistemic irrelevance (see e.g., \cite{krak2019hitting, skulj:09IJAR}).

\subsection{Imprecise Markov chains as sets of processes}

Each interpretation described above defines a family of \emph{compatible
	precise processes}, which need not be time-homogeneous or even
Markovian.  By a \emph{set of processes} we mean the collection of all
probability laws on (finite or infinite) state sequences whose one-step
transition matrices lie in \(\tmatrices\).

For a finite horizon \((X_{1:N})\) a compatible process is determined by
a probability mass function on \(\states^{N}\).  We denote the set of all such
\emph{compatible joint distributions} by
\[
\distr(X_{1:N})
:=\bigl\{\,p(x_{1:N}) : p
\text{ arises from some choice of transition matrices in }
\tmatrices\bigr\}.
\]

\paragraph{Sensitivity analysis (precise chains with uncertain
	parameters).}
Every compatible process is a homogeneous Markov chain with some
\(P\in\tmatrices\) and \(q^{(1)}\in\cset\):
\[
p^{\mathrm{s}}_{q^{(1)},P}(x_{1:N})
= q^{(1)}(x_{1})\,P(x_{1},x_{2})\cdots P(x_{N-1},x_{N}),
\qquad
\distr(X_{1:N})
=\bigl\{\,p^{\mathrm{s}}_{q^{(1)},P}
\mid  q^{(1)}\in\cset,\;P\in\tmatrices\bigr\}.
\]

\paragraph{Strong independence.}
At each time \(k\) a matrix \(P_{k}\in\tmatrices\) is chosen independently
of the past beyond \(X_{k}\).  A process is specified by
\(\Gamma=(q^{(1)},P_{1:N-1})\in\cset\times\tmatrices^{N-1}\) with joint
law
\[
p_{\Gamma}(x_{1:N})
= q^{(1)}(x_{1})\,P_{1}(x_{1},x_{2})\cdots
P_{N-1}(x_{N-1},x_{N}),
\qquad
\distr(X_{1:N})
=\bigl\{\,p_{\Gamma} \mid  \Gamma\in\cset\times\tmatrices^{\,N-1}\bigr\}.
\]

\paragraph{Epistemic irrelevance.}
Here the choice of transition probabilities may depend on the entire
history: for each past \(x_{1:k}\) one selects a probability vector
\(P_{k}(x_{1:k-1},x_{k},\cdot)\in\tmatrices_{x_k}\).  Writing
\(\Delta=(q^{(1)},P_{1:N-1})\) for such a collection, the joint mass
function is
\begin{equation}\label{eq:p-delta}
	p_{\Delta}(x_{1:N})
	= q^{(1)}(x_{1})\,
	P_{1}(x_{1},x_{2})\,\cdots\,
	P_{N-1}(x_{1:N-2},x_{N-1},x_{N}),
\end{equation}
and
\[
\distr(X_{1:N})
=\{\,p_{\Delta} \mid  \Delta\text{ ranges over all admissible sequences}\}.
\]

Across all three interpretations an imprecise Markov chain is the set
of its compatible precise processes.  Under the sensitivity-analysis and
strong-independence interpretations these processes remain Markovian, whereas
under epistemic irrelevance they generally do not.  This distinction is
crucial for time reversal: the reversal methods developed earlier rely
on the Markov property and therefore apply only to the first two
interpretations.

\subsection{Stationarity for imprecise Markov chains}

Let \((X_n)\) be an imprecise Markov chain determined by a transition
set \(\tmatrices\) and an initial credal set
\(\cset:=\distr(X_1)\).  By analogy with the precise case, the chain is
called \emph{stationary} when the marginal credal sets remain constant:
\[
\distr(X_k)=\distr(X_1),
\qquad k=2,3,\dots .
\]

\paragraph{Sensitivity interpretation.}
For the sensitivity analysis interpretation we have
\[
\distr(X_n)
=\{\,qP^{\,n-1} \mid  q\in\cset,\;P\in\tmatrices\},
\]
a set that is usually hard to characterise
\cite{krein:03, utkin:02}.  If every matrix in \(\tmatrices\) is
\emph{regular}, then as \(n\to\infty\) these sets converge to the
compact, convex set of stationary distributions of matrices in
\(\tmatrices\) \cite[Chap.~3]{hart:98}.

\paragraph{Strong independence and epistemic irrelevance.}
Under either interpretation a credal set \(\Pi\) is
\emph{stationary} when
\begin{equation}\label{eq:invariant-set}
	\Pi=\Pi\tmatrices:=\{\,qP \mid  q\in\Pi,\;P\in\tmatrices\}.
\end{equation}
A stationary set always exists \cite[Thm.~2]{skulj:09IJAR} but need not
be unique unless \(\tmatrices\) is \emph{regular} (all members are
regular).  A weaker condition guaranteeing uniqueness is given in
\cite{decooman-2008-a}; the general non-unique case is analysed in
\cite{skulj:13LAA}.

Consequently, an imprecise Markov chain can always be made stationary by
choosing an initial credal set that satisfies
\eqref{eq:invariant-set}.  If strict positivity is required, the chain
must satisfy the regularity (ergodicity) condition.

\subsection{Stationary imprecise Markov chains under strong independence via joint distribution matrices}\label{s-simcjdm}

In the usual forward representation an imprecise chain is specified by an
initial credal set \(\cset^{(1)}\) and a set of transition matrices
\(\tmatrices\).  Throughout this subsection we adopt the
\emph{strong-independence} interpretation: at each step \(k\) the chain
may employ any \(P_{k}\in\tmatrices\), independently of the past.  This
description is directional, propagating uncertainty only forward.

For tasks that demand time symmetry, such as stationarity and
reversibility, it is advantageous to work with the symmetric
\emph{joint matrix} representation based on the distribution of two
consecutive states.  We therefore switch to that view (still under strong
independence) and study its interaction with stationarity and time
reversal.

For any set \(\mathcal Q\) of two-step joint distribution matrices define
the associated \emph{left-} and \emph{right-marginal credal sets}
\begin{equation}\label{eq:general-left-right}
	\cset_{\mathcal Q}
	:=\{\mathbf 1\, Q^\top\mid Q\in\mathcal Q\},
	\qquad
	\cset_{\mathcal Q}^{*}
	:=\{\,\mathbf 1 Q \mid  Q\in\mathcal Q\}.
\end{equation}
Every \(Q\in\mathcal Q\) with \(\mathbf 1\, Q^\top>0\) determines a
transition matrix
\begin{equation}\label{eq:general-row-norm}
	P_{Q}:=\operatorname{diag}\!\bigl(\mathbf 1\, Q^\top\bigr)^{-1}Q,
\end{equation}
and we set \(\tmatrices_{\mathcal Q}:=\{P_{Q}\mid Q\in\mathcal Q\}\).

%-------------------- Application to each time step -------------
Conversely, given a forward description \((\cset^{(1)},\tmatrices)\) under
strong independence, put
\(\cset^{(k)}:=\cset^{(1)}\tmatrices^{\,k-1}\).
For each \(k=1,\dots ,N-1\) define
\begin{equation}\label{eq:def-Qk-set}
	\mathcal Q_{k}
	:=\{\operatorname{diag}(q)P \mid  q\in\cset^{(k)},\;P\in\tmatrices\}.
\end{equation}
Then \(\cset_{\mathcal Q_{k}}=\cset^{(k)}\),
\(\cset_{\mathcal Q_{k}}^{*}=\cset^{(k+1)}\)
(by \eqref{eq:general-left-right}), and
\(\tmatrices_{\mathcal Q_{k}}=\tmatrices\)
(by \eqref{eq:general-row-norm}).

The joint matrix model is inherently symmetric (transposition reverses
time), whereas the forward representation is not: reversing the relation
between \(\cset^{(k)}\) and \(\cset^{(k+1)}\) is generally cumbersome.

\begin{proposition}\label{prop:gamma-Q-general}
	Let \(\cset^{(1)}\) be an initial credal set and \(\tmatrices\) a set of
	transition matrices.  Fix \(N\ge 2\) and define
	\(\cset^{(k)}:=\cset^{(1)}\tmatrices^{\,k-1}\).  For each \(k\) set
	\(\mathcal Q_{k}:=\{\operatorname{diag}(q)P\mid q\in\cset^{(k)},\,P\in\tmatrices\}\).
	Every transition law
	\(\Gamma=(q^{(1)},P_{1:N-1})\) with \(q^{(1)}\in\cset^{(1)}\) and
	\(P_{k}\in\tmatrices\) generates a marginally compatible sequence
	\(Q_{1:N-1}\in\mathcal Q_{1:N-1}\) such that
	\(\Gamma_{Q_{1:N-1}}=\Gamma\).
\end{proposition}

\begin{proof}
	Direct consequence of Proposition~\ref{prop:p_Gamma}.
\end{proof}

\begin{remark}\label{rem:joint-strictly-more-general}
	The mapping
	\((\cset^{(1)},\tmatrices)\mapsto(\mathcal Q_{1},\dots ,\mathcal Q_{N-1})\)
	is not invertible: a sequence of joint sets fixes
	\(\cset^{(k)}=\cset_{\mathcal Q_{k}}\) and
	\(\tmatrices_{\mathcal Q_{k}}\) but need not satisfy
	\(q^{(k)}P_{k}\in\cset^{(k+1)}\) for \emph{all}
	\(q^{(k)}\in\cset^{(k)}\) and \(P_{k}\in\tmatrices_{\mathcal Q_{k}}\).
	Hence the joint matrix model is strictly more expressive than the
	forward one.
\end{remark}

\begin{example}\label{ex-joint-set-no-reverse}
	Let
	\[
	Q_{1}=\begin{pmatrix}0.1&0.2\\0.3&0.4\end{pmatrix},
	\qquad
	Q_{2}=\begin{pmatrix}0.2&0.2\\0.3&0.3\end{pmatrix},
	\qquad
	\mathcal Q=\operatorname{co}\{Q_{1},Q_{2}\}.
	\]
	Row sums are \(l_{1}=(0.3,0.7)\) and \(l_{2}=(0.4,0.6)\); column sums
	are \(r_{1}=(0.4,0.6)\) and \(r_{2}=(0.5,0.5)\), giving
	\[
	P_{1}=\begin{pmatrix}\tfrac13&\tfrac23\\[2pt]\tfrac37&\tfrac47\end{pmatrix},
	\qquad
	P_{2}=\begin{pmatrix}\tfrac12&\tfrac12\\[2pt]\tfrac12&\tfrac12\end{pmatrix}.
	\]
	But \(l_{2}P_{1}=(41/105,\,64/105)\approx(0.391,0.609)\) lies outside
	\(\operatorname{co}\{r_{1},r_{2}\}\), so \(\mathcal Q\) cannot arise from
	any single forward pair \((\cset,\tmatrices)\).
\end{example}

\begin{proposition}\label{prop:Q-constant-stationary}
	Let \(\cset\) be a credal set, \(\tmatrices\) a transition set, and
	\(\mathcal Q_{k}\) the associated joint distribution sets.
	\begin{enumerate}[(i)]
		\item If \(\cset\tmatrices=\cset\), then
		\(\mathcal Q_{k}\) is independent of \(k\).
		\item Conversely, if \(\mathcal Q_{k}\) is constant in \(k\), then
		\(\cset\tmatrices=\cset\); the chain is stationary in the forward
		sense as well.
	\end{enumerate}
\end{proposition}

\begin{proof}
	(i) If \(\cset^{(k+1)}=\cset^{(k)}\), then	\(\mathcal Q_{k+1}=\mathcal Q_{k}\) by \eqref{eq:def-Qk-set}.
	
	(ii) If \(\mathcal Q_{k}\) is constant, so are its row sums $\cset^{(k)}$.
\end{proof}

\medskip
\noindent
When the two-step joint set is constant, then clearly its row and column marginals
coincide.  Motivated by this symmetry we adopt
the following notion.

\begin{definition}[Joint-stationary chain]\label{def:joint-stationary}
	An imprecise Markov chain $(X_k)$ is said to be \emph{joint-stationary} if	there exists a single set of joint distribution matrices $\mathcal Q$ such that $\mathcal Q_k = \mathcal Q$ for all $k$. Equivalently, it represents the collection of precise processes whose two-step joint distribution matrices are all drawn from the same set $\mathcal Q$ at any time step. 
\end{definition}

Because the common joint set has equal row and column marginals, all
one-step marginal credal sets coincide:
\(\distr(X_{k})=\cset_{\mathcal Q}\) for every \(k\).
Hence Definition~\ref{def:joint-stationary} genuinely extends the usual
notion of stationarity based on forward transition sets. Note also that in the stationary case there is no need for horizon to be finite.

\section{Reversing imprecise Markov chains}\label{s-rimc}

Section~\ref{s-imc} showed that, regardless of interpretation, an imprecise Markov chain can always be viewed as a \emph{set of joint distribution laws}, Markovian or not.  We therefore begin by defining time reversal at this most general level—that is, for arbitrary sets of joint laws.  Once the general framework is in place, we specialise to the \emph{strong-independence} interpretation, where the symmetric joint matrix representation yields more detailed results.

\subsection{Reversal of joint laws}\label{s-rev-joint}
In this general setting, every stochastic process---precise or imprecise---can be represented by its set of joint distributions on finite paths.
We define time reversal directly on these joint laws, without reference to any particular generative construction.
The definition itself does not rely on the Markov property but instead works in full generality.
This general formulation will later be specialised to the reversal of imprecise Markov chains under strong independence.

Let \(X_{1:N}\) be a stochastic process on a finite state space
\(\states\) with joint distribution set
\(\distr(X_{1:N})\).

\begin{definition}
The \emph{reversal} of \(\distr(X_{1:N})\) is
\[
\distr^{*}(X_{1:N})
:=\bigl\{\,p^{*}\!\mid p\in\distr(X_{1:N})\bigr\},
\qquad
p^{*}(x_{1:N}):=p(x_{N:1}).
\]
\end{definition}
For \(N=2\) this reduces to
\(p^{*}(x_{1},x_{2})=p(x_{2},x_{1})\), i.e.~\(p^{*}=p^{\top}\).

\begin{proposition}\label{prop:rev-set}
	\(\distr^{*}(X_{1:N})=\distr(X_{N:1})\); hence the map
	\(p\mapsto p^{*}\) is a bijection between \(\distr(X_{1:N})\) and its
	reverse.
\end{proposition}

\begin{corollary}\label{cor:involution}
	Reversal is an \emph{involution}:
	\(
	\bigl(\distr^{*}(X_{1:N})\bigr)^{*}=\distr(X_{1:N})
	\).
\end{corollary}

\paragraph{Reversed gambles.}
For a real-valued function (or \emph{gamble})
\(f:\states^{N}\!\to\RR\), define \(f^{*}(x_{1:N}):=f(x_{N:1})\).

\begin{proposition}\label{prop:expect-rev}
	For every credal set \(\cset\subseteq\Delta(\states^{N})\) and gamble
	\(f\),
	\[
	\underline{E}_{\cset}[f^{*}]
	=\underline{E}_{\cset^{*}}[f],
	\qquad
	\overline{E}_{\cset}[f^{*}]
	=\overline{E}_{\cset^{*}}[f].
	\]
\end{proposition}

\begin{proof}
	For the lower expectation,
	\begin{align*}
	\underline{E}_{\cset}[f^{*}]
	& =\min_{p\in\cset}\sum_{x_{1:N}}p(x_{1:N})f(x_{N:1}) \\
	& = \min_{p\in\cset}\sum_{x_{1:N}}p(x_{N:1})f(x_{1:N}) \\
	& =\min_{p\in\cset}\sum_{x_{1:N}}p^{*}(x_{1:N})f(x_{1:N}) 
	=\underline{E}_{\cset^{*}}[f].
	\end{align*}
	The upper expectation case is analogous.
\end{proof}

\begin{proposition}\label{prop:marginal-rev}
	For each \(k=1,\dots,N\),
	\(
	\distr(X_{k})=\distr^{*}(X_{\,N-k+1})
	\).
\end{proposition}

\begin{corollary}\label{cor:rev-stationary}
	If \((X_{1:N})\) is stationary with common marginal credal set
	\(\Pi\), then the reversed process is also stationary with the same
	\(\Pi\).
\end{corollary}

\begin{definition}[Reversible finite horizon process]\label{def:rev-finite}
	A process \(X_{1:N}\) is \emph{reversible} if
	\(
	\distr^{*}(X_{1:N})=\distr(X_{1:N})
	\).
\end{definition}

Reversibility implies stationarity by
Corollary~\ref{cor:rev-stationary}, but not conversely.
For processes on an (infinite) time horizon we adopt the following definition (cf.\ \cite[§1.2]{kelly2011reversibility}).
\begin{definition}[Reversible infinite horizon process]\label{def:rev-infinite}
	A stochastic process $(X_n)_{n\in\mathbb Z}$ on 
	state space $\states$ is \emph{reversible} if, for every pair of
	integers $m\le n$, the finite subsequence
	$X_{m:n}$ is reversible in the sense of
	Definition~\ref{def:rev-finite}; that is,
	\[
	\distr^{*}(X_{m:n}) \;=\; \distr(X_{m:n}).
	\]
\end{definition}

\subsection{Reversing an imprecise Markov chain under the strong-independence interpretation}
An imprecise Markov chain can be viewed as the set of its compatible
precise processes, and, as shown above, its reverse is then the set of the reversed processes. In principle it is therefore possible to reverse any imprecise Markov chain. However, as we show in the following example, if a forward process is represented as an imprecise Markov chain in the form of a pair \((\cset,\tmatrices)\), the reversed process is generally not obtainable in the same form.

In the precise case it is well known that if a \(2\times2\) stochastic
matrix \(P\) is irreducible (and hence aperiodic), then the chain is
reversible and \(P_{\pi}^{*}=P\); see, for instance,
\cite[Example 8.29]{aldous:fill:14}.
The following example illustrates that, unlike in the precise case, the reversed imprecise process cannot generally be represented by a pair $(\Pi, \tmatrices)$ of the same form as the forward process. We construct a minimal two-state model to show that the stationary credal set for $\tmatrices$ is not preserved under reversal.

\begin{example}\label{ex-non-reversible}
	Consider the two transition matrices
	\[
	P_{1}=\begin{pmatrix}0.2&0.8\\0.7&0.3\end{pmatrix},
	\qquad
	P_{2}=\begin{pmatrix}0.6&0.4\\0.5&0.5\end{pmatrix},
	\qquad
	\tmatrices=\operatorname{co}\,\{P_{1},P_{2}\}.
	\]	
	We first compute the stationary credal set corresponding to \(\tmatrices\). Clearly every convex set closed for multiplication with the extreme points of $\tmatrices$ is also closed for multiplication with the entire convex set $\tmatrices$. Therefore, we analyse the extreme points of minimal such set that contains the stationary distributions of the extreme points. 
	Both matrices $P_1$ and $P_2$ are regular and therefore have a unique stationary
	distributions:
	\[
	\pi_{1}=(7/15,\,8/15)\approx(0.4667,0.5333),
	\qquad
	\pi_{2}=(5/9,\,4/9)\approx(0.5556,0.4444).
	\]
	Define
	\[
	u:=\pi_{2}P_{1}=(19/45,\,26/45)\approx(0.4222,0.5778),
	\qquad
	v:=\pi_{2}.
	\]
	Note that
	\(
	\pi_{1}P_{2}=(41/75,\,34/75)\approx(0.5466,0.4533)
	\)
	is a convex combination of \(\pi_{1}\) and \(\pi_{2}\), and that
	\(\pi_{2}P_{1}^{2}\) and \(\pi_{2}P_{1}P_{2}\) are convex combinations of
	\(u\) and \(v\).
	Because right multiplication by \(P_{1}\) or \(P_{2}\) is an affine map,
	every convex combination of \(u\) and \(v\) is sent to another
	convex combination of \(u\) and \(v\).
	Hence $\operatorname{co}\{u,v\}$ is closed for multiplication with $P_1$ and $P_2$, and therefore the stationary credal set for \(\tmatrices\) is exactly
	\[
	\Pi=\operatorname{co}\{u,v\}.
	\]
	Numerically, the first coordinate of any element of \(\Pi\) lies in the
	interval \([0.4222,\,0.5556]\).
	In general, the stationary credal set of an imprecise Markov chain need not be convex unless the set of transition matrices has separately specified rows, and would then only form a subset of its convex hull.
	In this two-state case, however, the stationary set is effectively one-dimensional.
	Because right multiplication by any $P\in \tmatrices$ maps every convex combination of the extreme points into another convex combination of the same points, the convex hull of these extremes is invariant under all transitions.
	Consequently, the stationary credal set must include, and indeed coincide with, all convex combinations of the two extreme stationary distributions.	
	
	Now compute the \(u\)-reverse of \(P_{1}\):
	\[
	(P_{1}^{*})_{u}=
	\begin{pmatrix}
		19/110 & 91/110 \\
		76/115 & 39/115
	\end{pmatrix}
	\approx
	\begin{pmatrix}
		0.1727 & 0.8273 \\
		0.6609 & 0.3391
	\end{pmatrix}.
	\]
	This matrix is not contained in \(\tmatrices\), so the set of
	\(\Pi\)-reverses,
	\[
	\tmatrices^{*}:=\{P_{q}^{*} \mid P\in\tmatrices,\;q\in\Pi\},
	\]
	strictly enlarges \(\tmatrices\). (Note that $\tmatrices\subseteq \tmatrices^{*}$ because every $P\in \tmatrices$ is reversible and therefore $P^*_\pi = P$, where $\pi$ is the corresponding stationary distribution.)
	
	Finally, we check whether \(\Pi\) might still be a stationary set for the enlarged transition set. Take \(w:=uP_{1}=(22/45,\,23/45)\approx(0.4889,0.5111)\); as expected,
	\(w(P_{1}^{*})_{u}=u\in\Pi\).
	However,
	\[
	v(P_{1}^{*})_{u}\approx(0.3897,0.6103),
	\]
	and since \(0.3897<0.4222\), this distribution lies outside \(\Pi\).
	Consequently, \(\Pi\) is not stationary for \(\tmatrices^{*}\).
	
	\medskip
	\noindent
	This two-state example highlights three fundamental difficulties with
	defining reversal in the ordinary \((\Pi,\tmatrices)\) framework:
	
	\begin{enumerate}[(i)]
		\item \textbf{Transition-set enlargement.}
		Because the \(u\)-reverse \((P_{1}^{*})_{u}\) is a legitimate
		transition matrix for the reversed process yet lies outside
		\(\tmatrices\), the reversed transition set must satisfy
		\(\tmatrices^{*}\supset\tmatrices\).
		Enlarging \(\tmatrices\) also enlarges its stationary credal
		set, so the reversed model is more imprecise than the
		forward one, violating the desired symmetry between the two
		directions and the involution property of the reversal operator.
		
		\item \textbf{Different stationary set.}
		Once \(\tmatrices^{*}\) is enlarged, the forward stationary set
		\(\Pi\) is no longer invariant, as demonstrated by finding
		\(v\in\Pi\) such that
		\(v(P_{1}^{*})_{u}\notin\Pi\).
		Hence \(\Pi\) cannot be stationary for the reversed process,
		contradicting Corollary~\ref{cor:rev-stationary},
		which states that the marginal credal sets should be preserved
		under time reversal.
		
		\item \textbf{Inherent issues.}
		One possible avenue to address the above difficulties would be to impose
		additional regularity conditions---for example, requiring interval bounds
		for transition matrices or that they have separately specified rows.
		However, these restrictions do not in general resolve the problem.
		A model satisfying such standard assumptions can be derived from the
		preceding example by defining the convex, separately specified credal set
		\[
		\tmatrices' = \{\, P \colon L \le P \le U \,\},
		\]
		where 
		\[
		L = \min\{P_1,P_2\} =
		\begin{pmatrix}
			0.2 & 0.4 \\ 0.5 & 0.3
		\end{pmatrix}, \quad
		U = \max\{P_1,P_2\} =
		\begin{pmatrix}
			0.6 & 0.8 \\ 0.7 & 0.5
		\end{pmatrix}.
		\]
		Clearly, $\tmatrices \subset \tmatrices'$, and hence the corresponding
		stationary sets satisfy $\Pi \subseteq \Pi'$ with $u \in \Pi'$.
		Nevertheless, the $u$-reverse of $P_1$ lies outside $\tmatrices'$,
		demonstrating that the difficulty described in~(i) persists even under
		these additional assumptions.
		This shows that the enlargement of the transition set under reversal is
		a structural property of the elementwise reversal operation itself,
		rather than an artefact of the simplified construction.
		
	\end{enumerate}
	
	Even in this minimal \(2\times2\) setting, a forward specification in
	terms of a pair \((\Pi,\tmatrices)\) cannot be reversed into another pair
	of the same form.  Therefore, a different representation is
	required to achieve true reversibility in imprecise Markov models.
\end{example}

In the next section we demonstrate that working directly with a
credal set of joint distribution matrices
avoids these problems and allows
reversibility without enlarging the model.

\subsection{Reversing an imprecise process given by a set of two-step joint distribution matrices}

Proposition \ref{prop:gamma-Q-general} and Example \ref{ex-joint-set-no-reverse}
show that representing (possibly non-homogeneous) Markov processes by a
\emph{sequence of joint distribution matrices} is strictly more expressive
than the usual forward model \((\cset,\tmatrices)\).
Example \ref{ex-non-reversible} further demonstrates that forward models are
not closed under reversal.

Because we are primarily interested in \emph{stationary} behaviour, and
Proposition \ref{prop:Q-constant-stationary} tells us that every stationary
imprecise Markov chain corresponds to a stationary process generated by a
\emph{single} joint distribution set, we adopt that framework.
Henceforth we assume the two-step sets
\(\mathcal Q_i\) are constant, and we therefore simply denote them by \(\mathcal Q\).
Allowing these sets to vary with time is possible in principle, but the
common practise in the literature on imprecise Markov chains is to keep the set fixed and let only its individual elements change.

\begin{theorem}\label{thm-reversibility-symmetry}
	Let \((X_{n})_{n\in \ZZ}\) be an imprecise Markov chain whose transition law is induced
	by a joint distribution set \(\mathcal Q\) satisfying
	\(\cset_{\mathcal Q}=\cset^{*}_{\mathcal Q}\).
	Then the chain is stationary with marginal credal set
	\(\cset_{\mathcal Q}\) and is reversible \emph{iff} \(\mathcal Q\) is
	symmetric, i.e.\ \(\mathcal Q^{\top}=\mathcal Q\).
\end{theorem}

\begin{proof}
	If \(\mathcal Q\) is symmetric, then every compatible process over arbitrary finite horizon $N$ has joint
	distribution law
	\[
	p_{\Gamma}(x_{1:N})
	=\frac{\displaystyle\prod_{k=1}^{N-1}Q_k(x_k,x_{k+1})}
	{\displaystyle\prod_{k=2}^{N-1}q^{(k)}(x_k)},
	\]
	with \(Q_{1:N-1}\in\mathcal Q^{N-1}\) marginally compatible and
	\(q^{(k)}=\mathbf 1Q_{k}\).  Its reverse is obtained by replacing each
	\(Q_k\) with \(Q_k^{\top}\), which still lies in \(\mathcal Q\); hence
	\(\distr^{*}(X_{1:N})=\distr(X_{1:N})\). Now since every finite horizon sequence is reversible, the infinite horizon process is also reversible (cf. Definition~\ref{def:rev-infinite}).
	
	Conversely, a \emph{two step} chain is reversible exactly when its joint distribution set
	is symmetric.  If the full chain is reversible, then every two-step subchain is
	reversible, and its set of joint distribution matrices is exactly \(\mathcal Q\), which then must
	be symmetric.
\end{proof}

\begin{remark}
	No positivity assumption on the marginal distributions is
		required for a set \(\mathcal Q\) to generate reversible
		processes.
\end{remark}

\subsection{Extending to a reversible model}

The following proposition states that every imprecise Markov chain can be extended to a reversible one.

\begin{proposition}
	Let $(X_{n})_{n\in\ZZ}$ be an imprecise Markov chain whose transition law is induced by a set of joint distribution matrices $\mathcal{Q}$ satisfying $\Pi = \csete_{\mathcal{Q}} = \csete^*_{\mathcal{Q}}$. 
	Then the set $\hat{\mathcal{Q}} = \mathcal{Q} \cup \mathcal{Q}^\top$ is the smallest set of joint distribution matrices that contains $\mathcal{Q}$ and induces a reversible imprecise Markov chain.  
	
	Moreover, the corresponding reversible process $(Y_{n})_{n\in\ZZ}$ is stationary with distribution set $\Pi$.
\end{proposition}

\begin{proof}
	The set $\hat{\mathcal{Q}}$ is clearly the smallest symmetric set containing $\mathcal{Q}$ and therefore the smallest inducing a reversible chain.  Transposing a joint distribution matrix swaps the left and right marginals, both contained in $\Pi$. Thus, adding transposed matrices does not increase the set of marginals.
\end{proof}

Next, we show that every symmetric set of joint distributions can be extended to a convex set, preserving symmetry. Convexity is desirable, as it ensures that the set of possible models is closed under convex combinations, which is useful for computational purposes.

\begin{proposition}
	Let $\mathcal{Q}$ be a symmetric set of joint distributions with the stationary set of distributions $\Pi$. Then $\co(\mathcal{Q})$ is a convex and symmetric set of joint distributions with the stationary set $\co(\Pi)$.
\end{proposition}

\begin{proof}
	To see that $\co(\mathcal{Q})$ is symmetric, take a convex combination $\alpha Q + \beta Q'$, where $0 \le \alpha, \beta \le 1$ and $\alpha + \beta = 1$. We have that $(\alpha Q + \beta Q')^\top = \alpha Q^\top + \beta Q'^\top$, which is a convex combination of the transposed matrices. By symmetry, $Q^\top$ and $Q'^\top$ are also in $\mathcal{Q}$, so $(\alpha Q + \beta Q')^\top \in \co(\mathcal{Q})$.

	Clearly, symmetry of $\mathcal Q$ implies that the corresponding unique stationary set $\Pi$ equals $\csete_{\mathcal Q} = \csete^*_{\mathcal Q}$, and clearly $\csete_{\co(\mathcal Q)} = \co(\csete_{\mathcal Q}) = \co(\Pi)$. 
\end{proof}
In the simplest form a convex set of joint distribution matrices is described by considering all joint distribution matrices that lie between a lower and upper bound $\low Q \le \up Q$. This model can be directly related to the \emph{probability interval} model \cite{augustin2014introduction, walley:91}, which is known to be non-empty whenever the sum of lower probabilities is less and the sum of upper probabilities more than 1. 

Note that the bounds themselves are not joint distribution matrices, however, the set of matrices 
\begin{equation}
	[\low Q, \up Q] := \{ Q\mid Q\text{ is a joint probability matrix }, \low Q\le Q\le \up Q \}
\end{equation}
is clearly a convex set of matrices, which we call an \emph{interval joint distribution matrix}. The following proposition is also obvious. 
\begin{proposition}
	Let $\low Q\le \up Q$ be bounds on the joint distribution matrices, such that $[\low Q, \up Q]$ is non-empty. Then if $\low Q = \low Q^\top$ and $\up Q = \up Q^\top$ both hold, then $[\low Q, \up Q]$ is symmetric. Moreover, if the bounds are exact---that is if they are reached by the elements of $[\low Q, \up Q]$---symmetry of $[\low Q, \up Q]$ implies symmetry of the bounds. 
\end{proposition}

\begin{example}
	Consider transition matrices from Example~\ref{ex-non-reversible}: 
	\[
	P_{1}=\begin{pmatrix}0.2&0.8\\0.7&0.3\end{pmatrix},
	\qquad
	P_{2}=\begin{pmatrix}0.6&0.4\\0.5&0.5\end{pmatrix},
	\qquad
	\tmatrices=\{P_{1},P_{2}\}.
	\]
	We construct the minimal reversible interval joint distribution matrix that contains the two reversible precise processes associated with $P_1$ and $P_2$. Using the corresponding stationary distributions 
	\[
	\pi_{1}=(7/15,\,8/15)\approx(0.4667,0.5333),
	\qquad
	\pi_{2}=(5/9,\,4/9)\approx(0.5556,0.4444).
	\]
	we obtain the joint matrices 
	\[
	Q_{1}=\diag(\pi_1) P_1 = \left(
	\begin{array}{cc}
		0.0933 & 0.3733 \\
		0.3733 & 0.16 \\
	\end{array}
	\right),
	\qquad
	Q_{2}=\diag(\pi_2) P_2 =\left(
	\begin{array}{cc}
		0.3333 & 0.2222 \\
		0.2222 & 0.2222 \\
	\end{array}
	\right).
	\]	
	The minimal interval joint distribution matrix $[\low Q, \up Q]$ containing both $Q_1$ and $Q_2$ is obtained by taking 
	\[
	\low Q = \left(
	\begin{array}{cc}
		0.0933 & 0.2222 \\
		0.2222 & 0.16 \\
	\end{array}
	\right),
	\qquad
	\up Q =\left(
	\begin{array}{cc}
		0.3333 & 0.3733 \\
		0.3733 & 0.2222 \\
	\end{array}
	\right),
	\]	
	which are clearly symmetric, and therefore induces a reversible imprecise Markov chain with symmetric interval joint distribution matrix $[\low Q, \up Q]$, and the reversible precise chains corresponding to $P_1$ and $P_2$ are among its compatible processes.  
\end{example}

\section{Application to random walks on weighted graphs}\label{s-rwwg}

Random walks on weighted graphs give a canonical, widely studied model
for reversible Markov chains.  The graph’s vertices form the state
space, and the transition probability from one vertex to another is
proportional to the weight of the connecting edge
\cite{aldous:fill:14,Feige1995,Gobel1974,Lovasz1993}.  Beyond sampling,
they underpin applications in network analysis
\cite{Coppersmith1993,lin2014mean,liu2012probabilistic,zhang2014effects},
social networks \cite{Backstrom2011,li2011link,yin2010unified}, and
recommender systems \cite{fouss2007random}.

Owing to this intuitive graph-based interpretation, random walks provide
a natural arena in which to apply the abstract reversibility results of
Section~\ref{s-rimc}.  On a \emph{precisely} weighted, undirected graph
the induced walk is automatically reversible: undirected edges ensure
that every path can be traversed in either direction with the same edge
weight, yielding symmetric transition probabilities.  Our aim here is to
show how an analogous symmetry arises when the edge weights are
specified only \emph{imprecisely}.  In this broader setting one must
decide whether to retain the strict undirected property or adopt a more
flexible requirement.  The dilemma mirrors an earlier choice in
the theory of imprecise Markov chains—whether to insist on time-homogeneity or
allow more general dynamics.  As there, demanding full symmetry of every
underlying precise model leads to an overly rigid, often unrealistic
description, whereas a more general formulation captures uncertainty
more faithfully.

\subsection{Precise random walk}

Let $\states$ be the vertex set of a connected \emph{undirected weighted
	graph} $\mathcal G=(\states,\mathcal E)$ with edge weights
$w(x,y)=w(y,x)\ge 0$, where $w(x,y)=0$ denotes absence of an edge.
Connectivity means that for every $u,v\in\states$ there exists a path
$u=x_0,x_1,\dots,x_n=v$ with $w(x_i,x_{i+1})>0$.  A weight function
$w\colon\mathcal E\to\RR_+$ induces the transition matrix
\begin{equation}\label{eq:rw-transition}
	P_w(x,y)=\frac{w(x,y)}{w(x)},
	\qquad
	w(x):=\sum_{z\in\states}w(x,z).
\end{equation}
The resulting \emph{random walk} $(X_n)_{n\ge 0}$ is time homogeneous.
It is ergodic when the greatest common divisor of the lengths of all
\emph{closed walks} through every vertex equals~1.  Its unique stationary
distribution is
\begin{equation}\label{eq:rw-stationary}
	\pi_w(x)=\frac{w(x)}{W},
	\qquad
	W:=\sum_{(x, y)\in\mathcal E}w(x, y).
\end{equation}
Equations~\eqref{eq:rw-transition} and \eqref{eq:rw-stationary} satisfy
the detailed-balance condition $\pi(x)P(x,y)=\pi(y)P(y,x)$, so the walk
is \emph{reversible} and all general results from Section~\ref{s-rimc}
apply.

\subsection{Interval-weight model}
Analogous to modelling imprecision in probability distributions—where
precise probabilities \(P(A)\) of certain events are replaced by
intervals \([\underline P(A),\overline P(A)]\)—the weights of a
graph’s edges can be made imprecise by introducing intervals
\([\underline w(x,y),\overline w(x,y)]\).
More generally, a precise weight function
\(w\colon\mathcal E\to\RR_+\) is replaced by two bounding functions
\(\underline w\) and \(\overline w\) with
\(\underline w\le\overline w\).
Because weights require no normalisation, such a pair induces the
\emph{convex} and non-empty set
\[
\mathcal W=\{\,w\mid \underline w\le w\le\overline w\},
\]
which contains both \(\underline w\) and \(\overline w\).

The corresponding set of transition matrices is
\[
\tmatrices_{\mathcal W}=\{P_w\mid w\in\mathcal W\}.
\]
This set, however, does \emph{not} inherit convexity from
\(\mathcal W\), as the following example shows.

\begin{example}
	Consider the weight functions
	\[
	W_1=\begin{pmatrix}1&3\\3&2\end{pmatrix},
	\qquad
	W_2=\begin{pmatrix}1&5\\5&2\end{pmatrix}.
	\]
	The induced transition matrices are
	\[
	P_1=\begin{pmatrix}\frac14&\frac34\\[2pt]\frac35&\frac25\end{pmatrix},
	\qquad
	P_2=\begin{pmatrix}\frac16&\frac56\\[2pt]\frac57&\frac27\end{pmatrix}.
	\]
	The convex combination
	\(W=\tfrac12W_1+\tfrac12W_2\) yields
	\[
	P=\begin{pmatrix}\frac15&\frac45\\[2pt]\frac46&\frac26\end{pmatrix},
	\]
	which is \emph{not} a convex combination of \(P_1\) and \(P_2\).
	Hence the set of transition matrices induced by an interval set of
	weights is, in general, non-convex.
\end{example}

Convexity is essential for efficient estimation of expectations via
linear programming.
To obtain a convex model, \v Skulj \cite{skulj:16IJAR} fixed the marginals
(the sums of edge weights incident to each vertex), resulting in convex
transition sets that, however, lack \emph{separately specified rows}.
Because the rows \(P(\cdot\mid x)\) cannot then be chosen
independently, the marginal sets at later times become non-convex,
precluding standard linear programming methods.

We therefore propose a different model that avoids fixing marginals and
instead represents the imprecise random walk through sets of joint
distribution matrices.
Recall that a precise random walk on a weighted graph is assumed to
start in its stationary distribution.

\begin{proposition}
	Let \(w\) be a weight function on a graph with vertex set
	\(\states\).
	The joint distribution of a two-step random walk starting from its
	stationary distribution is
	\[
	Q_w(x,y)=\frac{w(x,y)}{W},
	\]
	where \(W\) is the total weight \(W=\sum_{(x, y)\in \mathcal E}w(x, y)\) (each edge between different nodes counted
	twice).
\end{proposition}

\begin{proof}
	Since \(\pi_w(x)=w(x)/W\) and
	\(P_w(x,y)=w(x,y)/w(x)\), we have
	\(Q_w(x,y)=\pi_w(x)P_w(x,y)=w(x,y)/W\).
\end{proof}

Thus the set \(\mathcal W\) induces the set
\[
\mathcal Q_{\mathcal W}
=\bigl\{\,Q_w\mid \;Q_w(x,y)=w(x,y)/W,\;w\in\mathcal W\bigr\},
\]
with \(W=\sum_{(x, y)\in \mathcal E}w(x, y)\).
Unlike the transition set, \(\mathcal Q_{\mathcal W}\) is convex:

\begin{proposition}\label{prop-weights-convex}
	If \(\mathcal W\) is convex, then the corresponding set
	\(\mathcal Q_{\mathcal W}\) is convex.
\end{proposition}

\begin{proof}
	Let \(w,w'\in\mathcal W\) and \(\alpha,\beta\ge0\) with
	\(\alpha+\beta=1\).
	Writing \(W=\sum_{(x, y)\in\mathcal E}w(x, y)\) and \(W'=\sum_{(x, y)\in\mathcal E}w'(x, y)\), we obtain for all
	\(x,y\in\states\)
	\[
	Q_{\alpha w+\beta w'}(x,y)
	=\frac{\alpha w(x,y)+\beta w'(x,y)}{\alpha W+\beta W'}
	=A Q_w(x,y)+B Q_{w'}(x,y),
	\]
	where \(A=\alpha W/(\alpha W+\beta W')\) and
	\(B=\beta W'/(\alpha W+\beta W')\).
	Because \(A,B\ge0\) and \(A+B=1\), \(Q_{\alpha w+\beta w'}\) is a
	convex combination of \(Q_w\) and \(Q_{w'}\).
	
	Conversely, given any $A, B \ge 0$ with $A + B = 1$, define
	\[
	\alpha = \frac{A/W}{A/W + B/W'}, 
	\quad 
	\beta = \frac{B/W'}{A/W + B/W'}.
	\]
	Then $\alpha, \beta \ge 0$, $\alpha + \beta = 1$, and substituting these into the expressions above yields the same $A$ and $B$. 
	Hence, there is a one-to-one correspondence between $(\alpha, \beta)$ and $(A, B)$.
	
	Since $\alpha w + \beta w' \in \mathcal{W}$ by convexity of $\mathcal{W}$, it follows that $Q_{\alpha w + \beta w'} \in \mathcal{Q}_{\mathcal{W}}$. 
	Therefore, every convex combination $A Q_w + B Q_{w'}$ of elements of $\mathcal{Q}_{\mathcal{W}}$ again lies in $\mathcal{Q}_{\mathcal{W}}$, proving that $\mathcal{Q}_{\mathcal{W}}$ is convex.
\end{proof}

As explained earlier, processes determined by a set of joint distribution
matrices arise from marginally compatible sequences
\(Q_{1:N-1}\subset\mathcal Q_{\mathcal W}\).
For an \emph{undirected} graph the left and right marginals of each
\(Q\) coincide, so admissible sequences would consist only of matrices
with identical marginals, which is a severe restriction.
To obtain a richer model we therefore relax symmetry of the weight
functions and consider directed graphs.

\subsection{Random walks on directed weighted graphs with interval weights}

We now consider convex sets $\mathcal{W}$ of not necessarily symmetric
weight functions defined on a complete graph with vertex set
$\states$.  For each $w\in\mathcal{W}$, let $W=\sum_{(x, y)\in \mathcal E}w(x, y)$ denote the total weight (the sum of all directed edges). Because every edge contributes exactly once, $W$ equals both the total outgoing and incoming weight.  For each such
$w$ we define the joint distribution matrix $Q_w = w/W$.  The left
marginal distribution is $\tfrac{w_O}{W}$, where
$w_O(x)=\sum_{y\in\states} w(x,y)$ is the total outgoing weight from
vertex $x$, while the right marginal is $\tfrac{w_I}{W}$, with
$w_I(y)=\sum_{x\in\states} w(x,y)$ the total incoming weight to vertex
$y$.  A random walk on the directed graph is the process
$(X_{n})_{n\in \ZZ}$ that moves from $x$ to $y$ with probability
$w(x,y)/w_O(x)$.  Let
$\mathcal{Q}_{\mathcal{W}}=\{\,Q_w\mid  w\in\mathcal{W}\,\}$ denote the
set of all joint matrices induced by $\mathcal{W}$. The left and right marginals are then 
\begin{align*}
	\mathcal M_{\mathcal Q_{\mathcal{W}}} = \left\{ \frac{w_O}{W} \mid  w\in\mathcal{W} \right\}, \qquad \mathcal M^*_{\mathcal Q_{\mathcal{W}}} = \left\{ \frac{w_I}{W} \mid  w\in\mathcal{W} \right\}.
\end{align*} 
Let $w \in \mathcal W$ be a weight function and $w^\top$ its transpose: $w^\top(x, y) = w(y, x)$. Then clearly, $w^\top_I = w_O$ and $w^\top_O = w_I$. Let $\mathcal W^\top = \{w^\top\mid w\in \mathcal W \}$ denote the transposed set of weights; we then have  
\begin{equation*}
	\mathcal M_{\mathcal Q_{\mathcal{W^\top}}} = \mathcal M^*_{\mathcal Q_{\mathcal{W}}} \qquad \mathcal M^*_{\mathcal Q_{\mathcal{W^\top}}} = \mathcal M_{\mathcal Q_{\mathcal{W}}}.
\end{equation*}
We say that $\mathcal W$ is \emph{symmetric} if $\mathcal W^\top = \mathcal W$. It is now clear that a symmetric weight function induces identical left and right set of marginals  $\mathcal M_{\mathcal Q_{\mathcal{W}}} = \mathcal M^*_{\mathcal Q_{\mathcal{W}}}$. 

The convexity argument for symmetric weights extends verbatim to the directed case, since it relies only on normalization and not on symmetry. 
\begin{proposition}\label{prop-weights-non-symmetric-convex}
	If \(\mathcal W\) is a convex set of not necessarily symmetric weight functions, then the corresponding set
	\(\mathcal Q_{\mathcal W}\) is convex.
\end{proposition}

\begin{definition}
	Let $\mathcal{W}$ be a set of weight functions on a complete graph
	with vertex set $\states$, and let
	$\mathcal{Q}_{\mathcal{W}}$ be the corresponding set of joint
	distribution matrices.  A \emph{random walk} on the weighted graph
	with weight set $\mathcal{W}$ is the imprecise Markov chain whose finite-horizon
	transition laws are
	\[
	\mathcal{G}_{\mathcal{W}}
	=\{\,\Gamma_{Q_{1:N-1}} \mid  Q_i\in\mathcal{Q}_{\mathcal{W}}\ \text{for all }i,
	\text{ and the } Q_i \text{ are pairwise marginally compatible}\}.
	\]
\end{definition}

The next result links symmetry of the weight set to symmetry of the
corresponding joint matrices.

\begin{proposition}
	Let $\mathcal{W}$ be a convex and symmetric set of weight functions.
	Then the corresponding set of joint matrices $\mathcal{Q}_{\mathcal{W}}$
	is symmetric, and the random walk generated from it is reversible.
\end{proposition}

\begin{proof}
	For every $w\in\mathcal{W}$ we have $Q_w^{\top}=Q_{w^\top}$. 
	Since $\mathcal{W}$ is symmetric, $w^\top\in\mathcal{W}$ and therefore 
	$Q_w^{\top}\in\mathcal{Q}_{\mathcal{W}}$, implying that 
	$\mathcal{Q}_{\mathcal{W}}$ is closed under transposition and thus symmetric. Since the induced sets of left and right marginals respectively are identical, Theorem~\ref{thm-reversibility-symmetry} implies that	the random walk induced by~$\mathcal{W}$ is reversible.
\end{proof}

A symmetric set $\mathcal{W}$ may contain individual weight functions
that are \emph{not} symmetric—edge weights can depend on travel
direction—but the set as a whole is symmetric.  This mirrors the case
of an imprecise time-homogeneous Markov chain, which contains
non-homogeneous precise processes yet remains homogeneous when viewed
as an imprecise model.

\begin{remark}
	Convexity of the weight set $\mathcal{W}$ implies convexity of the	corresponding set of joint matrices
	$\mathcal{Q}_{\mathcal{W}}$.	However, an interval specification	$\mathcal{W}=[\underline W,\overline W]$
	does not in general yield a componentwise interval set	$\mathcal{Q}_{\mathcal{W}}=[\underline Q,\overline Q]$.	
	For each component $(x,y)$, the attainable lower and upper bounds of
	$Q_w(x,y)$ over the interval set of weights $\mathcal{W}$ are obtained at opposite
	corners of $\mathcal{W}$:
	\[
	\underline Q(x, y)
	=\frac{\underline W(x, y)}{
		\underline W(x, y)
		+ \sum_{(u,v)\neq(x,y)} \overline W(u, v)},
	\qquad
	\overline Q(x, y)
	=\frac{\overline W_{xy}}{
		\overline W(x, y)
		+ \sum_{(u,v)\neq(x,y)} \underline W(u, v)}.
	\]
	These bounds provide the tightest componentwise enclosure, but
	not every matrix $Q$ satisfying
	$\underline Q\le Q\le \overline Q$ can be realised by a weight function
	$w\in[\underline W,\overline W]$. 	Hence, $\mathcal{Q}_{\mathcal{W}}$ is a convex polytope that is
	generally a strict subset of the interval joint distribution matrix 
	$[\underline Q,\overline Q]$.
\end{remark}

\section{Numerical calculations}\label{s-nc}
Because reversibility implies stationarity, numerical analysis shifts
from estimating marginal distributions of the variables $X_k$
to estimating their joint distributions.
These joint distributions fully determine all probabilistic aspects
of a reversible imprecise Markov chain.

As is standard in imprecise probability models,
imprecision in a joint distribution is expressed
through the range of expected values of a functional $f(X_{1:N})$.
In particular, we compute the \emph{lower} and \emph{upper expectations}
\[
\low E(f)
= \min_{p\in \csete} \sum_{x_{1:N}} p(x_{1:N}) f(x_{1:N}),
\qquad
\up E(f)
= \max_{p\in \csete} \sum_{x_{1:N}} p(x_{1:N}) f(x_{1:N}),
\]
where $\csete$ denotes the set of all joint distributions
$p_\Gamma$ generated by transition laws
$\Gamma = \Gamma_{Q_{1:N-1}}$ with $Q_k \in \mathcal{Q}_k$ for all $k$.

\subsection{Multilinear optimization formulation}
For notational simplicity we introduce elementwise lower and upper bounds
\[
\underline Q,\, \overline Q \in \mathbb{R}_{\ge 0}^{s\times s},
\qquad
\underline Q \le \overline Q,
\]
which define the interval-valued credal set
\[
\mathcal Q
= \{\, Q \mid \underline Q \le Q \le \overline Q \,\}.
\]
More general credal sets can be incorporated by replacing these interval constraints
with linear inequalities of the corresponding form, but same computational principle applies.

For a real-valued functional $f\colon \mathcal{X}^N \to \mathbb{R}$,
the lower and upper expectations can be obtained by solving
the following \emph{multilinear program}:
\[
\begin{array}{ll}
	\text{maximise/minimise} &
	\displaystyle
	E(f)
	= \sum_{x_{1:N}} f(x_{1:N})
	\frac{\prod_{k=1}^{N-1} Q_k(x_k,x_{k+1})}
	{\prod_{k=2}^{N-1} q^{(k)}(x_k)} \\[3ex]
	\text{subject to} &
	\underline Q \le Q_k \le \overline Q,
	\quad k=1,\dots,N-1, \\[1ex]
	&
	\sum_{x_{k-1}} Q_{k-1}(x_{k-1},x_k)
	= \sum_{x_{k+1}} Q_k(x_k,x_{k+1}),
	\quad k=2,\dots,N-1, \\[1ex]
	&
	Q_k(x_k,x_{k+1}) \ge 0,
	\quad
	\sum_{x_k,x_{k+1}} Q_k(x_k,x_{k+1}) = 1.
\end{array}
\]
(The non-negativity requirement could be omitted in this interval case, as it is imposed by $\underline Q\ge 0$, however it is needed in the case of more general constraints defining $\mathcal Q$.)
The objective function is generally nonlinear, but becomes \emph{linear}
in any individual matrix $Q_k$ when all others are fixed,
subject to the compatibility constraints.
In this case, all marginals $q^{(k)}$ are fixed,
and we can define
\[
\varphi(x_k, x_{k+1})
= \sum_{x_{1:k-1}} \sum_{x_{k+2:N}}
\frac{\displaystyle
	\prod_{l=1}^{k-1} Q_l(x_l,x_{l+1})
	\prod_{m=k+1}^{N-1} Q_m(x_m,x_{m+1})
}{
	\displaystyle
	\prod_{r=2}^{N-1} q^{(r)}(x_r)
}\,
f(x_{1:N}),
\]
which represents the aggregated contribution of all fixed transitions
to the expectation of $f$.
Then the optimization with respect to $Q_k$ reduces to
\[
\begin{array}{ll}
	\text{maximise/minimise} &
	\displaystyle
	E^k_Q(\varphi)
	= \sum_{x_k, x_{k+1}} Q(x_k, x_{k+1})\, \varphi(x_k, x_{k+1}) \\[3ex]
	\text{subject to} &
	\underline Q \le Q \le \overline Q, \\[1ex]
	&
	\sum_{x_{k+1}} Q(x_{k},x_{k+1}) = q^{(k)}(x_{k}), \\[1ex]
	&
	\sum_{x_{k}} Q(x_k,x_{k+1}) = q^{(k+1)}(x_{k+1}), \\[1ex]
	&
	Q(x_k,x_{k+1}) \ge 0,
	\quad
	\sum_{x_k,x_{k+1}} Q(x_k,x_{k+1}) = 1.
\end{array}
\]
This is a standard linear program in the variables $Q(x_k,x_{k+1})$,
since both the objective and the constraints are linear in $Q$. (Note that one of the marginal constraints is omitted in the cases $k = 1, k = N-1$.)

This implies that, at any global extremal solution $(Q_1^*,\dots, Q_{N-1}^*)$ corresponding to a functional $f$, each $Q_k^*$ is an optimal vertex of its feasible slice, which is a linear program. Hence, the global extreme can always be realised at an extreme point of the entire feasible polytope, even though optimizing each $Q_k$ separately would generally yield only a local extreme.

The total number of variables grows as $O((N-1)|\mathcal{X}|^2)$, since each matrix $Q_k$ contains $|\mathcal{X}|^2$ entries, while the number of linear compatibility constraints increases as
$O((N-2)|\mathcal{X}|)$. Hence, the overall problem scales linearly with the time horizon $N$
but quadratically with the number of states.

In principle, such problems can be approached through alternating linear programs or by evaluating the objective over vertex combinations, taking advantage of the fact that global optima occur at extreme configurations of the admissible joint matrices. However, these methods become computationally demanding as the number of states or time steps increases, since the number of extreme points grows exponentially. Consequently, direct enumeration or full alternating optimization is feasible only for small-scale examples such as the one considered below, while larger systems require additional structural simplifications or approximate solution methods. More generally, nonlinear programs of this type can be addressed with standard global optimization techniques—such as branch-and-bound, convex relaxations, or stochastic search—though these are typically computationally expensive for high-dimensional multilinear objectives (see, e.g., \cite{Floudas2000,HorstTuy1996}).

\subsection{Numerical example}

To illustrate the multilinear formulation, consider a two-state
reversible imprecise Markov chain over three time steps
$(X_1, X_2, X_3)$.
Each local joint distribution matrix is parametrised as
\[
Q_k =
\begin{pmatrix}
	0.1+\alpha_k & 0.2+\beta_k \\
	0.2+\beta_k & 0.5 - \alpha_k - 2\beta_k
\end{pmatrix},
\qquad
\alpha_k, \beta_k \in [0,0.1], \quad k=1,2.
\]
Each matrix $Q_k$ represents a convex set of admissible joint
distributions, whose extreme points are obtained by assigning the
boundary values $0$ or $0.1$ to $\alpha_k$ and $\beta_k$.
This yields four extreme points for each $Q_k$, and therefore
$4 \times 4 = 16$ combinations in the product space of $(Q_1,Q_2)$.
However, not all pairs are feasible: marginal compatibility requires
\[
\sum_{x_1} Q_1(x_1, y)
= \sum_{x_3} Q_2(y, x_3)
\quad\Longleftrightarrow\quad
\alpha_1+\beta_1 = \alpha_2+\beta_2,
\]
which restricts the number of admissible vertex combinations.

We consider the event
$f(x_1,x_2,x_3)=\mathbf{1}\{x_1=x_3\}$,
representing the probability of returning to the same state after two
steps. For fixed joint matrices, this probability equals
\[
P(X_1=X_3)
= \sum_{x_2}
\frac{
	\sum_{x_1=x_3} Q_1(x_1,x_2)\, Q_2(x_2,x_3)
}{q^{(2)}(x_2)},
\qquad
q^{(2)}(y) = \sum_x Q_1(x,y).
\]
Substituting the parametric form of $Q_k$ gives
\begin{align*}
	F(\alpha_1,\beta_1,\alpha_2,\beta_2)
	& = 
	\frac{(0.1+\alpha_1)(0.1+\alpha_2) + (0.2+\beta_1)(0.2+\beta_2)}
	{0.3+\alpha_1+\beta_1} \\[0.5ex]
	& \quad +
	\frac{(0.2+\beta_1)(0.2+\beta_2)
		+ (0.5-\alpha_1-2\beta_1)(0.5-\alpha_2-2\beta_2)}
	{0.7-\alpha_1-\beta_1}.
\end{align*}
Because the objective function is bilinear in
$(Q_1,Q_2)$, extrema are attained at the extreme points of the feasible
set. Hence, it suffices to evaluate $F$ at the vertex combinations
$\alpha_k,\beta_k \in \{0,0.1\}$ that also satisfy the compatibility
condition $\alpha_1+\beta_1=\alpha_2+\beta_2$.
The admissible combinations and corresponding probabilities
$P(X_1=X_3)=F(\alpha_1,\beta_1,\alpha_2,\beta_2)$
are listed in Table~\ref{tab:extreme_points}.

\begin{table}[h!]
	\centering
	\caption{Admissible extreme points and corresponding values of
		$F(\alpha_1,\beta_1,\alpha_2,\beta_2)$.}
	\label{tab:extreme_points}
	\begin{tabular}{ccccc}
		\toprule
		\# &
		$(\alpha_1,\beta_1)$ &
		$(\alpha_2,\beta_2)$ &
		$q^{(2)}$ &
		$P(X_1=X_3)$ \\
		\midrule
		1 & (0.0, 0.0) & (0.0, 0.0) & (0.3, 0.7) & $0.58095$ \\
		2 & (0.0, 0.1) & (0.0, 0.1) & (0.4, 0.6) & $0.55$ \\
		3 & (0.0, 0.1) & (0.1, 0.0) & (0.4, 0.6) & $0.50$ \\
		4 & (0.1, 0.0) & (0.0, 0.1) & (0.4, 0.6) & $0.50$ \\
		5 & (0.1, 0.0) & (0.1, 0.0) & (0.3, 0.7) & $0.53333$ \\
		6 & (0.1, 0.1) & (0.1, 0.1) & (0.4, 0.6) & $0.52$ \\
		\bottomrule
	\end{tabular}
\end{table}

The smallest and largest return probabilities are therefore
\[
\underline{P}(X_1 = X_3) = 0.50,
\qquad
\overline{P}(X_1 = X_3) = 0.58095.
\]
The upper bound corresponds to the symmetric stationary configuration
where 
\[Q_1=Q_2=
\begin{pmatrix}0.1&0.2\\0.2&0.5\end{pmatrix}, \]
while the lower bound arises from the most asymmetric pair of compatible
matrices. This confirms that the extreme values of the multilinear
program are attained at the extreme points of the joint feasible
polytope $(Q_1,Q_2)$.

%========================================================
\section{Conclusion and outlook}\label{s-co}
%========================================================
The main contribution of this paper is a systematic treatment of
time-reversal for imprecise Markov chains under the
\emph{strong-independence} interpretation.
We began by describing all \emph{compatible} precise processes, all of which are 
(inhomogeneous) Markov chains, and then reversed them individually; the
reversed imprecise process is obtained by collecting these reversed
precise chains.
Because a forward representation \((\mathcal Q,\mathcal P)\) is generally
unavailable for the reversed family, we adopted an inherently symmetric
description based on \emph{sets of joint distribution (edge-measure)
	matrices}.
Although the forward and joint distribution views are equivalent for
precise chains with strictly positive marginals, the joint distribution
approach is strictly more general in the imprecise settings and permits a direct time-reversal.
In this setting a reversible process, like its classical counterpart,
is fully characterised by two-step dependence; the difference is that
this dependence is now encoded by sets of joint distributions rather
than by marginal distributions coupled with conditional (transition)
matrices.
Note that sets of joint matrices can naturally be taken to be \emph{closed} and \emph{convex}, in which case they form credal sets on the product space $\states\times\states$ and can be specified by linear-programming constraints.

The motivation for introducing reversible imprecise Markov chains is primarily conceptual rather than algorithmic.
They extend the classical notion of reversibility to settings where transition probabilities are only partially specified, allowing robust analysis of symmetric stochastic systems under uncertainty.
Although current computational methods are limited to small systems, the joint distribution representation developed here establishes a foundation for future algorithmic advances and for studying robustness in domains where reversibility naturally occurs, such as network flows and random walks on graphs.

Three directions appear particularly promising for future work. First, extending the joint distribution framework to continuous-time settings would widen its applicability to birth–death processes and other models based on transition rates. Second, we need more efficient algorithms for evaluating lower and upper expectations of multiple-step path functionals, which are essential for inference with reversible imprecise chains. Finally, we aim to apply the theory to concrete domains—such as robust queueing, interval-valued epidemic models, and stochastic-network analysis—to demonstrate its practical value.

\section*{Acknowledgments}

\begin{enumerate}
	\item The author acknowledges the financial support of the Slovenian Research Agency
	(research core funding No.~P5-0168).
	\item The author thanks the two anonymous reviewers for their constructive and insightful comments, which helped improve the clarity and overall quality of the paper.
\end{enumerate}

	% BIBLIOGRAPHY
\bibliographystyle{plain}
%\bibliography{../../../references/references_all}

\end{document}